\documentclass[12pt,leqno,oneside]{amsart}
\usepackage[width=6.5in,height=8.8in]{geometry}

\usepackage{mathrsfs,dsfont}
\usepackage{amsmath,amstext,amsthm,amssymb,bbm}
\usepackage{comment}
\usepackage{charter}
\usepackage{typearea}
\usepackage{pdfsync}

\def\bt{\begin{thm}}
\def\et{\end{thm}}
\def\bl{\begin{lem}}
\def\el{\end{lem}}
\def\bd{\begin{defi}}
\def\ed{\end{defi}}
\def\bc{\begin{cor}}
\def\ec{\end{cor}}
\def\bp{\begin{proof}}
\def\ep{\end{proof}}
\def\br{\begin{rem}}
\def\er{\end{rem}}

\newtheorem{thm}{Theorem}[section]
\newtheorem{prop}[thm]{Proposition}
\newtheorem{lem}[thm]{Lemma}
\newtheorem{defn}[thm]{Definition}

\newtheorem{rem}[thm]{Remark}
\newtheorem{cor}[thm]{Corollary}

\numberwithin{equation}{section}

\newcommand{\C}{\Bbb{C}}
\newcommand{\R}{\Bbb{R}}

\newcommand{\E}{\Bbb{E}}

\newcommand{\Er}{\Bbb{E}N_n}

\newcommand{\bthm}{\begin{thm}}
\newcommand{\ethm}{\end{thm}}
\newcommand{\bstp}{\begin{stp}}
\newcommand{\estp}{\end{stp}}
\newcommand{\blemma}{\begin{lemma}}
\newcommand{\elemma}{\end{lemma}}
\newcommand{\bprop}{\begin{prop}}
\newcommand{\eprop}{\end{prop}}
\newcommand{\bpf}{\begin{pf}}
\newcommand{\epf}{\end{pf}}
\newcommand{\bdefn}{\begin{defn}}
\newcommand{\edefn}{\end{defn}}
\newcommand{\brk}{\begin{rmrk}}
\newcommand{\erk}{\end{rmrk}}
\newcommand{\bcrl}{\begin{crl}}
\newcommand{\ecrl}{\end{crl}}

\title[Expected number of random real roots]{Expected number of real roots for random linear combinations of  orthogonal polynomials associated with radial weights}
\author{Turgay Bayraktar}


\begin{document}

\begin{abstract}
In this note, we obtain asymptotic expected number of real zeros for random polynomials of the form
$$f_n(z)=\sum_{j=0}^na^n_jc^n_jz^j$$ where $a^n_j$ are independent and identically distributed real random variables with bounded $(2+\delta)$th absolute moment and the deterministic numbers $c^n_j$ are normalizing constants for the monomials $z^j$ within a weighted $L^2$-space induced by a radial weight function satisfying suitable smoothness and growth conditions. 
\end{abstract}

\maketitle
\section{Introduction}
In this note we consider random univariate polynomials of the form
\begin{equation}\label{rp}
f^{\zeta}_n(z)=\sum_{j=0}^na^n_jc^n_jz^j
\end{equation} where $c^n_j$ are deterministic constants and $a^n_j$ are independent copies of a real valued non-degenerate random variable $\zeta$ of mean zero and variance one. We denote the number of real zeros of $f^{\zeta}_n$ by $N_n^{\zeta}.$ Therefore $N_n^{\zeta}:\mathcal{P}_n\to \{0,\dots,n\}$ defines a random variable on the set of polynomials $\mathcal{P}_n$ of degree at most $n$. We write $f_n$ and $N_n$ for short when $\zeta=N_{\R}(0,1)$ is the real Gaussian. 

The study of the number of real roots for Kac polynomials (i.e. $c^n_j=1$ for all $j$ and $n$) goes back to Bloch and P\'olya \cite{BPo} where they considered the case the random variable $\zeta$ is the uniform distribution on the set $\{-1,0,1\}$. This problem has been considered by Littlewood \& Offord in a series of papers \cite{LO, LO1,LO2} for real Gaussian, Bernoulli and uniform distributions. In \cite{Kac} Kac established a remarkable formula for the expected number of real zeros of Gaussian random polynomials 
\begin{equation}\label{kac1}
\E N_n=\frac{4}{\pi}\int_0^1\frac{\sqrt{A(x)C(x)-B^2(x)}}{A(x)}dx
\end{equation}
where $$A(x)=\sum_{j=0}^nx^{2j},\ B(x)=\sum_{j=1}^njx^{2j-1},\ C(x)=\sum_{j=1}^nj^2x^{2j-2}$$
and (\ref{kac1}) in turn implies that
\begin{equation}\label{kac2}
\E N_n=\big(\frac{2}{\pi}+o(1)\big)\log n.
\end{equation}
Later, Erd\"{o}s and Turan \cite{ET} obtained more precise estimates. The result stated in (\ref{kac2}) was also generalized to Bernoulli distributions by Erd\"{o}s \& Offord \cite{EO} as well as to distributions in the domain of attraction of the normal law by Ibragimov \& Maslova \cite{IM1,IM2}.

On the other hand, for models other than Kac ensembles, the behavior of $N_n$ changes considerably. In \cite{EdK}  Edelman \& Kostlan gave a beautiful geometric argument in order to calculate expected number of real roots of random polynomials. The argument in \cite{EdK} applies in a quite general setting of random sums of the form 
$$f_n(z)=\sum_{j=0}^na_jP^n_j(z)$$ where $P^n_j$ are entire functions that take real values on the real line (see also \cite{vanderbei, LPX}  for recent treatment of this problem). In particular, Edelman and Kostlan \cite[\S 3]{EdK} proved that 
 \[\E N_n=\begin{cases}  \sqrt{n} & \text{for Elliptic polynomials (i.e. $c^n_j=\sqrt{{n \choose j}}$)} \\ 
\big(\frac{2}{\pi}+o(1)\big)\sqrt{n}  & \text{for Weyl polynomials (i.e. $c^n_j=\sqrt{\frac{1}{j!}}$)}. \end{cases} \]

More recently, Tao and Vu \cite{TaoVu2} established some local universality results concerning the correlation
functions of the zeroes of random polynomials of the form (\ref{rp}). In particular, the results of \cite{TaoVu2} generalized aforementioned ones for real Gaussians to the setting where $\zeta$ is a random variable satisfying the moment condition $\E|\zeta|^{2+\delta}<\infty$ for some $\delta>0.$  Moreover, Tao and Vu also obtained some variance estimate on $N_n$ which leads to a weak law of large numbers for real zeros.

We remark that all the three models above (Kac, elliptic and Weyl polynomials) arise in the context of orthogonal polynomials. In this direction, the expected distribution of real zeros for random linear combination of Legendre polynomials is studied by Das \cite{Das71} in which he proved that 
$$\E N_n=\frac{n}{\sqrt{3}}+o(n).$$ Recently, Lubinsky, Pritsker \& Xie \cite{LPX, PX, LPX1} generalized this result to the random orthogonal polynomials with suitable weights supported on the real line.

In this note we study real roots of random polynomials which are random linear combinations of orthogonal polynomials induced by a non-negative radially symmetric (i.e. $\varphi(z)=\varphi(|z|)$) $\mathscr{C}^2$-weight function satisfying the growth condition
\begin{equation}\label{growth}
\varphi(z)\geq (1+\epsilon)\log |z|\ \text{for}\ |z|\gg1
\end{equation}
for some fixed $\epsilon>0.$ In the present setting, one can define associated equilibrium potential $U_{\varphi}$ and weighted equilibrium measure by $\mu_e:=\frac{1}{2\pi}\Delta U_{\varphi}$ where the latter is supported on the compact set  $D_{\varphi}:=\{z\in \C: U_{\varphi}(z)=\varphi(z)\}$ (see \S2 for details). The interior $\mathcal{B}_{\varphi}:=\{z\in D_{\varphi}: \Delta\varphi(z)>0\}$ is often called as \textit{Bulk} in the literature.
 
We define a weighted $L^2$-norm on the space $\mathcal{P}_n$ of polynomials of degree at most $n$ by 
\begin{equation}\label{n}
\|f_n\|^2_{\varphi}:=\int_{\C}|f_n(z)|^2e^{-2n\varphi(z)}dz
\end{equation} where $dz$ denotes the Lebesgue measure on $\C$. In particular, the monomials $\{z^j\}_{j=0}^n$ form an orthogonal basis for $\mathcal{P}_n.$ We consider random polynomials of the form (\ref{rp})
where $a^n_j$ are independent copies of a real random variable $\zeta$ and 
$$c^n_j:=(\int_{\C}|z|^{2j}e^{-2n\varphi(z)}dz)^{-\frac12}.$$ 

Our main result (Theorem \ref{main}) establishes the asymptotics of expected number $\E N_n^{\zeta}$ of real roots for random polynomials with independent identically distributed (iid) random coefficients satisfying mild moment conditions: 

\begin{thm}\label{main}
Let $\varphi:\C\to \R$ be a non-negative radially symmetric weight function of class $\mathscr{C}^2$ satisfying (\ref{growth}) and $\zeta$ be a non-degenerate real random variable of mean zero and variance one satisfying $\E|\zeta|^{2+\delta}<\infty$ for some $\delta>0.$ Then the expected number of real zeros for random polynomials $f_{n}^{\zeta}(z)$ satisfies  
\begin{equation}\label{azeros}
\lim_{n\to \infty}\frac{1}{\sqrt{n}}\Er^{\zeta}=\frac{1}{\pi}\int_{\mathcal{B}_{\varphi}\cap \R}\sqrt{\frac12\Delta\varphi(x)}dx
\end{equation}
\end{thm}

In order to prove Theorem \ref{main} for real Gaussian $\zeta=N_{\R}(0,1)$, we utilize Edelman and Kostlan's \cite{EdK} approach together with near and off diagonal Bergman kernel asymptotics obtained in \cite{Berman,B9}. Then the assertion of Theorem \ref{main} is a consequence of Theorem \ref{thm2} which indicates that with high probability (complex) zeros of a random polynomial $f_n$ accumulate in the bulk. For non-Gaussian case, we adapt the replacement principle of \cite{TaoVu2} to our setting in order to prove that $\E N_n^{\zeta}$ has universal property in the sense that under mild moment assumptions the $\lim_{n\to \infty}\frac{1}{\sqrt{n}}\E N_n^{\zeta}$ is independent of the choice of the random variable $\zeta.$ 
 
  In the special case, $\varphi(z)=\frac{|z|^2}{2}$ we obtain $c_j^n=\sqrt{\frac{n^{j+1}}{\pi j!}}$ which in turn gives rise to scaled Weyl polynomials and we recover results of Edelman \& Kostlan \cite{EdK}  and Tao \& Vu \cite{TaoVu2}. In \S4, we provide a more general family of circular weights and calculate the asymptotic expected number of real zeros for the corresponding ensemble of random polynomials.

Finally, I would like to thank the referee whose comments improved the exposition.

\section{Background}
Let $\varphi:\C\to\R$ be a $\mathscr{C}^{2}$ function satisfying 
\begin{equation}\label{gr}
\varphi(z)\geq (1+\epsilon)\log|z| \ \text{for} \ |z|\gg1
\end{equation} for some $\epsilon>0.$ In what follows we denote the set of all subharmonic functions on $\C$ by $SH(\C)$ and we let $L(\C)$ denote the class of all $u\in SH(\C)$ with the property that $u(z)-\log|z|$ is bounded above as $|z|\to \infty.$ One can define the \textit{equilibrium potential}
$$U_{\varphi}(z):=\sup\{u(z): u \in L(\C),\ \text{and}\ u\leq \varphi\ \text{on}\ \C\}.$$
It is well known that $U_{\varphi}(z)-\log^+|z|$ is bounded on $\C,$ where $\log^+|z|=\max(0,\log|z|)$ (see \cite[\S1.4]{SaffTotik} for details). 
Throughout this note we assume that $\varphi(z)$ is non-negative and radially symmetric (i.e. $\varphi(z)=\varphi(|z|)$ for $z\in \C$). 

In this special setting there is at least one other way of obtaining $U_{\varphi}$ as follows. We may define \\
$$\Phi:\R\to\R$$
$$\Phi(s):=\varphi(z)$$ where $s=\log|z|$ denotes logarithmic coordinate. Then it turns out that
\begin{equation}\label{eqp}
\Phi_{\varphi}(s):=\sup\{v(s): v\ \text{is convex,}\ \|\frac{dv}{ds}\|_{\infty}\leq 1\ \text{and}\ v\leq \Phi\ \text{on} \ \R\}
\end{equation}
satisfies $\Phi_{\varphi}(s)=U_{\varphi}(z)$ for $s=\log|z|.$ Here, the condition $\|\frac{dv}{ds}\|_{\infty}\leq 1$ ensures that the upper envelope $U_{\varphi}$ has logarithmic growth at infinity. In particular, the graph of $U_{\varphi}$ is the convex hull that of $\varphi.$

The \textit{weighted equilibrium measure} is defined as Laplacian (in the sense of distributions) of the equilibrium potential 
$$\mu_e:=\frac{1}{2\pi}\Delta U_{\varphi}.$$ 
 It follows from \cite{Berman} that  the measure $\mu_e$ is supported on the compact set  $$D_{\varphi}=\{z\in \C: U_{\varphi}(z)=\varphi(z)\}.$$
We also remark that $\mu_e$ is the unique minimizer of the weighted energy functional
$$I_{\varphi}(\nu):=-\int \int \log|z-w| d\nu(z)d\nu(w)+2\int \varphi(z) d\nu(z)$$ over all probability measures supported on $D_{\varphi}$ (see \cite[Theorem 3.1.3]{SaffTotik}). Following \cite{SaffTotik}, the constant $F_\varphi:=I(\mu_e)-\int\varphi d\mu_e$ is called the \textit{modified Robin constant} for $\varphi.$ 

The inner product
\begin{equation*}
\langle f,g\rangle=\int_{\C}f(z)\overline{g(z)}e^{-2n\varphi(z)}dz
\end{equation*}
induces an $L^2$-norm, denoted by $\|\cdot\|_{\varphi}$, on the vector space of polynomials $\mathcal{P}_n$ of degree at most $n.$ Since $\varphi$ is radially symmetric $P_j^n(z)=c^n_jz^j$ form an orthonormal basis with respect to this norm when $c_j^n:=(\|z^j\|_{\varphi})^{-1}.$ It is well known that (cf. \cite{SaffTotik})
\begin{equation}\label{coef}
\lim_{n\to \infty}\frac1n\log|c_n^n|=F_{\varphi}. 
\end{equation}

The next proposition is of independent interest and will be useful in the sequel. It implies that the equilibrium potential can be obtained as upper envelope $$U_{\varphi}(z)=(\limsup_{j\to \infty}\frac{1}{n}\log|P^n_j(z)|)^*$$ where $P_j^n(z)=c_j^nz^j$ denote orthonormal polynomials and $u^*(z):=\displaystyle\limsup_{w\to z}u(w)$ denotes the upper semi-continuous regularization of the function $u.$ 
\begin{prop}\label{orthogonal}
For every $\epsilon>0$ and $z\in \C\setminus \{0\}$ there exists $\delta>0$ such that for sufficiently large $n$
$$\#\{j\in\{0,\dots,n\}: c^n_j|z|^j>e^{n(U_{\varphi}(z)-3\epsilon)}\}\geq \delta n.$$
\end{prop}
\begin{proof}
We denote the probability measures $\mu_n=\frac{1}{a_n}e^{-2n\varphi(z)}dz$ where $a_n:=\int_{\C}e^{-2n\varphi(z)}dz$.  Since $\varphi(z)$ is non-negative, the sequence of measures $\{\mu_n\}_{n=1}^{\infty}$ satisfies large deviation principle (LDP) on $\C$ with the rate function $2\varphi(z)$ (see e.g. \cite[1.1.5]{DeSt}). Next, we define $$c^n_{tn}:=(\int_{\C}|z|^{2nt}e^{-2n\varphi(z)}dz)^{-\frac12}$$ for $t\in[0,1].$ Then by Varadhan's lemma \cite[Theorem 2.1.10]{DeSt} and (\ref{gr}), for every $t\in [0,1]$
\begin{eqnarray*}
-\lim_{n\to \infty}\frac1n\log c^n_{tn} &=&\sup_{r>0}(t\log r-\varphi(r))\\
&=& \sup_{s\in \R}(ts-\varphi(e^s))\\
&=:&u(t)
\end{eqnarray*}
where $u(t)$ (being  Legendre-Fenchel transform of $\varphi$) is a lower-semicontinuous convex function. Moreover,  Legendre-Fenchel transform of $u(t)$ coincides with convex hull of $\Phi(s)$ which is nothing but $\Phi_{\varphi}(s)$ defined in (\ref{eqp}). Thus, for every $s\in \R$ there exists $t_0\in [0,1]$ such that
$$st_0-u(t_0)>\Phi_{\varphi}(s)-\epsilon$$ and by lower-semicontinuity there exists an interval $\mathcal{J}\subset [0,1]$ containing $t_0$ such that
$$st-u(t)>\Phi_{\varphi}(s)-2\epsilon\ \ \text{for}\ t\in \mathcal{J}.$$ This implies that for sufficiently large $n$ and $s=\log|z|$ we have
$$\frac1n \log (c^n_{tn}|z|^{nt})>U_{\varphi}(z)-3\epsilon$$ for $t\in \mathcal{J}.$ 
Now, letting $\mathcal{J}_n:=\{j\in \{0,\dots,n\}: \frac{j}{n}\in \mathcal{J}\}$ we obtain $$\#\mathcal{J}_n\geq \frac{n}{2}|\mathcal{J}|$$ for sufficiently large $n.$
\end{proof}
\subsection{Bergman kernel asymptotics}
Growth condition (\ref{gr}) ensures that every polynomial $f(z)$ of degree at most $n$ has finite weighted $L^2$-norm defined by (\ref{n}). We denote the corresponding Hilbert space of polynomials by $(\mathcal{P}_n,\|\cdot\|_{\varphi}).$ For any given orthonormal basis $\{P_j^n\}_{j=0}^n$ of $\mathcal{P}_n$ the \textit{Bergman kernel} of $\mathcal{P}_n$ may be defined as
$$K_n(z,w)=\sum_{j=0}^{n}P_j^n(z)\overline{P_j^n(w)}.$$
We also denote its derivatives 
$$K_n^{(1,0)}(z,w)=\sum_{j=0}^n(P_j^n(z))'\overline{P_j^n(w)}\ \text{and}\ K_n^{(1,1)}(z,w)=\sum_{j=0}^n(P_j^n(z))'\overline{(P_j^n(w))'}. $$
 For fixed $z_0\in \mathcal{B}_{\varphi}$, we let $g:\C\to \C$ be the holomorphic function  $$g(z)=\varphi(z_0)+2\frac{\partial\varphi}{\partial z}(z_0)(z-z_0)$$ and we define 
 $$I_n(u,v):=\exp(-n[g(z_0+\frac{u}{\sqrt{n}})+\overline{g(z_0+\frac{v}{\sqrt{n}})}]).$$
We studied Bergman kernel asymptotics in \cite{B9} (cf. \cite{Christ,Lindholm,Delin,Berman}). The following near diagonal asymptotics of Bergman kernel was obtained in \cite[Theorem 2.3]{B9}:
\begin{thm}\label{nda}
Let $\varphi:\C\to \R$ be a $\mathscr{C}^2$ weight function satisfying (\ref{gr}) and $z\in \mathcal{B}_{\varphi}$ be fixed point. Then 
\begin{equation}\label{nda1} \frac1nK_n(z+\frac{u}{\sqrt{n}},z+\frac{v}{\sqrt{n}})I_n(u,v)\to \frac{1}{2\pi}\Delta\varphi(z)\exp(\frac12\Delta\varphi(z)u\overline{v})\ \text{as}\ n\to \infty
\end{equation}
in $C^{\infty}$-topology on compact subsets of $\C_u\times \C_v.$
\end{thm}	
Next, we use Theorem \ref{nda} in order to obtain diagonal asymptotics of the functions $K_n^{(1,0)}$ and $K_n^{(1,1)}:$

\begin{thm}\label{corelation}
Let $\varphi:\C\to \R$ be a $\mathscr{C}^2$ weight function satisfying (\ref{gr}) then
\begin{equation}\label{bkd1}
n^{-2}K_n^{(1,0)}(z,z)e^{-2n\varphi(z)}\to \frac{1}{\pi}\Delta\varphi(z)\frac{\partial \varphi}{\partial z}(z)
\end{equation}
and
\begin{equation}\label{bkd2}
n^{-3}K_n^{(1,1)}(z,z)e^{-2n\varphi(z)}\to \frac{1}{\pi}\Delta\varphi(z)\big(2\frac{\partial \varphi}{\partial z}(z)\frac{\partial \varphi}{\partial \overline{z}}(z)\big)
\end{equation}
uniformly on compact subsets of $\mathcal{B}_{\varphi}.$ In particular, if $\varphi$ is radially symmetric then for every real interval  $[a,b]\Subset \mathcal{B}_{\varphi}$
\begin{equation}\label{corelation1}
\frac{1}{\sqrt{n}}\frac{\sqrt{K_n^{(1,1)}(x,x)K_n(x,x)-(K_n^{(1,0)}(x,x))^2}}{K_n(x,x)}\to \sqrt{\frac{1}{2}\Delta\varphi(x)}
\end{equation}
uniformly on $[a,b].$
\end{thm}
\begin{proof}
We fix a compact subset $K\subset \mathcal{B}_{\varphi}$ and fix $z\in K.$ Then by Theorem \ref{nda} we have 
\begin{equation}\label{derivative}
\frac{\partial}{\partial u}[n^{-1}K_n(z+\frac{u}{\sqrt{n}},z+\frac{v}{\sqrt{n}})I_n(u,v)]\to \frac{1}{4\pi}(\Delta\varphi(z))^2\overline{v}\exp(\frac12\Delta\varphi(z)u\overline{v})
\end{equation}
as $n\to \infty.$ Note that the left hand side of (\ref{derivative}) is given by
 \begin{eqnarray*}
 \frac{I_n(u,v)}{\sqrt{n}}[\frac1nK_n^{(1,0)}(z+\frac{u}{\sqrt{n}},z_0+\frac{v}{\sqrt{n}})-2K_n(z+\frac{u}{\sqrt{n}},z+\frac{v}{\sqrt{n}})\frac{\partial\varphi}{\partial z}(z)]
 \end{eqnarray*}
 In particular, we have
 \begin{equation}\label{nda2}
n^{-2}K_n^{(1,0)}(z+\frac{u}{\sqrt{n}},z+\frac{v}{\sqrt{n}})I_n(u,v)\to \frac{1}{\pi}\Delta\varphi(z)\frac{\partial\varphi}{\partial z}(z)\exp(\frac12\Delta\varphi(z)u\overline{v})
\end{equation}
in $C^{\infty}$-topology on compact subsets of $\C_u\times \C_v.$ Then by covering $K$ with finitely many balls of radius $\frac{1}{\sqrt{n}}$ we deduce (\ref{bkd1}). Next, passing to $\frac{\partial}{\partial \overline{v}}$ on both sides of (\ref{nda2}) and using (\ref{nda2}) we see that
\begin{eqnarray}\label{nda3}
 \frac{1}{n^3}K_n^{(1,1)}(z+\frac{u}{\sqrt{n}},z+\frac{v}{\sqrt{n}})I_n(u,v)  &=&
 \frac{2}{\pi}\Delta\varphi(z)\frac{\partial\varphi}{\partial z}(z)\frac{\partial\varphi}{\partial \overline{z}}(z)\exp(\frac12\Delta\varphi(z)u\overline{v})\nonumber\\
&+&\frac{1}{n\pi} \big[\frac{1}{4}(\Delta\varphi(z))^2+(\frac12\Delta\varphi(z))^3u\overline{v}\big]\exp(\frac12\Delta\varphi(z)u\overline{v})+o(1)
\end{eqnarray}
as $n\to\infty.$ Hence, we infer (\ref{bkd2}).

Finally, using (\ref{nda1}), (\ref{nda2}), (\ref{nda3}) and the assumption $\varphi$ is radially symmetric we obtain (\ref{corelation1}).
\end{proof}

\subsection{Distribution of complex zeros}  In this section we assume that $a_j^n$ are independent copies of the real Gaussian $N_{\R}(0,1).$ We denote the number of (complex) zeros of a random polynomial $f_n$ in an open set $U$ by $\mathcal{N}_{U}(f_n).$ The main result of this section indicates that with high probability almost all zeros of a random polynomial are contained in the bulk $\mathcal{B}_{\varphi}:$
\begin{thm}\label{thm2}
Let $\epsilon>0$ and $U\Subset \C$ such that $\partial U$ has zero Lebesgue measure. Then
\begin{equation}
Prob_n\big[f_n:|\frac1n\mathcal{N}_{U}(f_n)-\int_{U}\Delta U_{\varphi}(z)|>\epsilon]\leq \exp(-C_{U,\epsilon} n) \ \text{for}\ n\gg1.
\end{equation} 
\end{thm} 
 In the setting of Gaussian analytic functions, Sodin and Tsirelson \cite{SoT2} obtained the analogue result (see also \cite{SZZ} for the setting of holomorphic sections). The results of \cite{SoT2,SZZ} are stronger in the sense that the speed of convergence is $n^2.$ Here, we give a simpler argument at the cost of losing the optimal estimate in the speed. First, we prove a preparative lemma:
\begin{lem}\label{lem1}
Let $\epsilon>0$ and $U\Subset \C.$ Then for sufficiently large $n$ there exists $E_n\subset \mathcal{P}_n$ such that $Prob_n(E_n)\leq e^{-C_{\epsilon}n}$ and 
$$|\log|f(z)|-U_{\varphi}(z)|\leq \epsilon$$ for every $f\in \mathcal{P}_n\setminus E_n$ and $z\in \overline{U}$
\end{lem}
\begin{proof}
It follows from \cite[\S6]{BloomL} that $\frac{1}{2n}\log K_n(z,z)$ converges to $U_{\varphi}(z)$ locally uniformly on $\C.$ In particular, we have $$|\frac{1}{2n}\log K_n(z,z)-U_{\varphi}(z)|\leq \frac{\epsilon}{2}$$ for $n\gg1$ and $z\in\overline{U}.$
Then we write
\begin{eqnarray*}
\log|f(z)| &=& \log\frac{|f(z)|\sqrt{K_n(z,z)}}{\sqrt{K_n(z,z)}}\\
&=&\log|\langle a,u(z)\rangle|+\frac{1}{2n}\log K_n(z,z)
\end{eqnarray*} where $u(z)=(u_j(z))_{j=0}^n$ is a unit vector in $\C^{n+1}$ and $a=(a_j)_{j=0}^n$ is the ``random coefficient vector" of $f(z).$ By triangular inequality we have
$$|\frac1n\log|f(z)|-U_{\varphi}(z)|\leq \frac1n|\log|\langle a,u(z)\rangle||+\frac{\epsilon}{2}.$$ 
Thus, it is enough to show that
$$Prob_n[\frac1n|\log|\langle a,u(z)\rangle||>\frac{\epsilon}{2}]\leq \exp(-C_{\epsilon}n)$$ for every $z\in \overline{U}.$
First, we obtain the upper bound:
\begin{eqnarray*}
Prob_n[\frac1n\log|\langle a,u(z)\rangle|>\frac{\epsilon}{2}] &\leq & Prob_n[(a_j)_{j=0}^n: \|a\|>\exp(\frac12\epsilon n)]\\
&\leq& (n+1)Prob[|\zeta|>\exp(\frac{\epsilon n}{2})] \\
& = & O\big((n+1)\exp(-\exp(\epsilon n))\big)
\end{eqnarray*}
where we used $\zeta= N_{\R}(0,1).$

On the other hand, since $u(z)$ is a unit vector we have the small ball probability of Gaussian random variable $\sum_{j=0}^na_ju_j(z)$
$$Prob_n[|\langle a,u(z)\rangle|<\exp(-\frac{\epsilon n}{2})]= O(\exp(-\epsilon n)).$$
which finishes the proof.
\end{proof}

\begin{proof}[Proof of Theorem \ref{thm2}]
Fix $\epsilon>0$ and let $\chi_U$ denote indicator function of $U$. Since $\mu_e$ is absolutely continuous with respect to Lebesgue measure we can find smooth functions $\psi^{\pm}:\C\to\R$ with compact support such that $0\leq \psi^-\leq\chi_U\leq \psi^+\leq 1$ satisfying 
$$\mu_e(U)-\epsilon \leq \int_{\C}\psi^-d\mu_e\ \text{and}\ \int_{\C}\psi^+d\mu_e\leq \mu_e(\overline{U})+\epsilon.  $$
 Then applying Lemma \ref{lem1} on $supp(\psi^+)$ we obtain  
\begin{eqnarray*}
\frac1n\mathcal{N}_{U}(f_n)\leq \frac1n\sum_{\{\zeta_i: f_n(\zeta_i)=0\}} \psi^+(\zeta_i) &= & \frac{1}{2\pi}\int_{\C}\frac1n\log|f_n(z)|\Delta(\psi^+(z))\\
&\leq& \frac{1}{2\pi} \int_{\C} U_{\varphi}(z)\Delta(\psi^+(z))+\epsilon\\
&\leq & \mu_e(\overline{U})+2\epsilon 
\end{eqnarray*} for $f\in \mathcal{P}_n\setminus E_n.$

One can obtain the lower estimate by using a similar argument with $\psi^-$.  
\end{proof}

\section{Proof of Theorem \ref{main}}
\subsection{Gaussian case}
 It follows from \cite{EdK} (see also \cite{vanderbei,LPX}) that for every measurable set $E\subset \R$ the expected number of real zeros of $f_n$ in $E$ is given by
\begin{equation}\label{kac}
 \E N_n(E)=\int_{E}g_n(x)dx
 \end{equation}
 where 
\begin{equation}\label{gn} 
g_n(x)=\frac{1}{\pi}\frac{\sqrt{K_n(x,x)K_n^{(1,1)}(x,x)-\big(K_n^{(0,1)}(x,x)\big)^2}}{K_n(x,x)}
\end{equation}
 and 
 $$K_n(x,x)=\sum_{j=0}^n(c_j^n)^2x^{2j},\ K^{(0,1)}(x,x)=\sum_{j=0}^nj(c_j^n)^2x^{2j-1},\ K^{(1,1)}(x,x)=\sum_{j=0}^nj^2(c_j^n)^2x^{2j-2} $$
 for $x\in \R.$
 Moreover, it follows from Theorem \ref{corelation} that for every real interval $[a,b]\Subset \mathcal{B}_{\varphi}$
$$ g_n(x)=\frac{1}{\pi}\sqrt{n}\big(\sqrt{\frac12\Delta\varphi(x)}+o(1)\big)\ \text{as}\ n\to \infty$$
uniformly on $[a,b].$

 Hence, the assertion follows from additivity of $\E N_n$ on $\R$, absolute continuity of $\E N_n$ with respect to Lebesgue measure $dx$ and Theorem \ref{thm2}.


 \subsection{Non-Gaussian case}
In what follows we let $D_r(z)$ denote a disc in complex plane of radius r  and center $z.$ For a complex random variable $\eta$ we denote its concentration function by
$$\mathcal{Q}(\eta,r):=\sup_{z\in \Bbb{\C}}Prob\{\eta\in D_r(z)\}. $$  We say that $\eta$ is \textit{non-degenerate} if $\mathcal{Q}(\eta,r)<1$ for some $r>0.$ If $\eta$ and $\xi$ are independent complex random variables and $r,c>0$ then we have 
\begin{equation}\label{levi}
\mathcal{Q}(\eta+\xi,r)\leq \mathcal{Q}(\eta,r)\ \text{and} \ \mathcal{Q}(c\zeta,r)=\mathcal{Q}(\zeta,\frac{r}{c}).
\end{equation}

  In order to prove the statement of Theorem \ref{main} for non-Gaussian ensembles we use the replacement principle of  \cite[Theorem 3.1]{TaoVu2}. We verify the hypotheses of \cite[Theorem 3.1]{TaoVu2} by proving sufficient conditions introduced in \cite[\S4]{TaoVu2}. \\
  
  \textit{Concentration of log-magnitude.} First, we prove the following result which verifies  \cite[Theorem 3.1(i)]{TaoVu2}.
  \begin{prop}\label{concentration}
   For fixed $\epsilon>0, z_0\in \C$ and $r>0$
  $$Prob_n[|\frac1n\log|f_n^{\zeta}(z)|-U_{\varphi}(z)|>\epsilon]=O(\frac{1}{\sqrt{n}})$$
  for every $z\in D_r(z_0)\setminus\{0\}$ and sufficiently large $n.$ 
  \end{prop}

  \begin{proof} In order to get the upper bound we modify  the argument given in Lemma \ref{lem1} by using the same notation:
  \begin{eqnarray*}
  Prob_n[|f_n(z)|>e^{n(U_{\varphi(z)}+\epsilon)}] &\leq & Prob_n[\frac1n\log|\langle a,u(z)\rangle|>2\epsilon]\\
  &=& Prob_n[|\langle a,u(z)\rangle|>e^{2n\epsilon}]\\
  &\leq& (n+1) Prob[|\zeta|>e^{2n\epsilon}]\\
  &\leq & C_{\delta}(n+1)e^{-4n\epsilon}
  \end{eqnarray*} 
 where we used $\E|\zeta|^{2+\delta}<\infty.$
 
 Next, we prove the lower bound by using an idea from \cite{KZ2}. By Proposition \ref{orthogonal} that for every $\epsilon>0$ there exists an interval $\mathcal{J}\subset [0,1]$ such that
 $$c^n_j|z|^j>e^{n(U_{\varphi}-\epsilon)}$$ for $j\in \mathcal{J}_n:=\{1\leq i\leq n: \frac{i}{n}\in \mathcal{J}\}.$ 
 Next, we define
 $$X_n=\sum_{j\in \mathcal{J}_n}\alpha^n_ja^n_j\ \text{and}\ Y_n=\sum_{j\not\in \mathcal{J}_n}\alpha^n_ja^n_j$$
where $$\alpha^n_j=e^{-n(U_{\varphi}-\epsilon)}c^n_j|z|^j.$$ Then by (\ref{levi}) and sufficiently large $n$ we have
 \begin{eqnarray*}
 Prob_n[|f_n(z)|<e^{n(U_{\varphi}(z)-2\epsilon)}]\leq \mathcal{Q}(X_n+Y_n, e^{-\epsilon n})\leq \mathcal{Q}(X_n,e^{-\epsilon n}).
 \end{eqnarray*}
 Now, it follows from Kolmogorov-Rogozin inequality \cite{Esseen} and $\alpha^n_j>1$ for $j\in \mathcal{J}_n$ that
 $$ \mathcal{Q}(X_n,e^{-\epsilon n})\leq C(\sum_{j\in \mathcal{J}_n}(1- \mathcal{Q}(\alpha^n_ja^n_j,e^{-\epsilon n}))^{-\frac12}\leq C_1 |\mathcal{J}_n|^{-\frac12}\leq C_2n^{-\frac12}.$$
  \end{proof}
  Next, combining with the argument given in the proof of Theorem \ref{thm2}, it follows from Proposition \ref{concentration} that
  $$Prob_n[|\frac1n\mathcal{N}_{D_r(z_0)}-\mu_e(D_r(z_0))|>\epsilon]=O(\frac{1}{\sqrt{n}}).$$  This together with \cite[Proposition 4.1]{TaoVu2} give \cite[Theorem 3.1(ii)]{TaoVu2}.\\

  \textit{Comparibility of log-magnitude.}  Note that Proposition \ref{orthogonal} and \cite[Theorem 4.5]{TaoVu2} imply  \cite[Theorem 3.1(iii)]{TaoVu2}.\\
  
  \textit{Repulsion of zeros.} In order to complete the proof of non-Gaussian case it remains to prove weak repulsion of zeros for $\zeta=N_{\R}(0,1).$ For any natural numbers $k,l$ following \cite[\S1]{TaoVu2} we define the the \textit{mixed $(k, l)$-correlation function} $$\rho_f^{(k,l)}:\R^k\times (\C\setminus \R)^l\to \R^+$$ 
    \begin{eqnarray*}
  & & \E\sum_{i_1,\dots,i_k}\sum_{j_1,\dots,j_l}\psi(\zeta_{i_1},\dots,\zeta_{i_k},\eta_{j_1},\dots,\eta_{j_l}) \\
  &=& \int_{\R^k}\int_{(\C\setminus \R)^l}\psi(x_1,\dots,x_k,z_1,\dots,z_l)\rho_f^{(k,l)}(x_1,\dots,x_k,z_1,\dots,z_l)dx_1\dots dx_k dz_1\dots dz_l
  \end{eqnarray*}
  where $\zeta_1,\dots \zeta_k$ denote the real roots and $\eta_1,\dots \eta_k$ denote the complex roots of Gaussian random polynomial $f_n.$ In particular, the function $\rho_{f_n}^{(1,0)}(x)$ is the density of expected number of real zeros $\E N_N.$
  
  It follows from \cite[Lemma 4.6]{TaoVu2} that for fixed $x\in\mathcal{B}_{\varphi}\cap\R$ it is enough to prove the real repulsion estimate
  \begin{equation}\label{corre1}\rho_{f_n}^{(2,0)}(x,x+\frac{1}{\sqrt{n}})=O(\frac{1}{\sqrt{n}})\ \text{as}\ n\to \infty
  \end{equation} and the complex repulsion estimate
  \begin{equation}\label{corre2}\rho_{f_n}^{(0,1)}(x+\sqrt{-1}\frac{1}{\sqrt{n}})=O(\frac{1}{\sqrt{n}}) \ \text{as}\ n\to \infty
  \end{equation} hold for the random polynomial $f_n$. As the estimates (\ref{corre1}) and (\ref{corre2}) are standard in the literature and follow from Kac-Rice formula and Bergman kernel asymptotics in \S2 (see e.g. the argument given in \cite[\S11]{TaoVu2}, see also \cite[\S 7]{BDi}) we omit the details. 
\section{examples}
In this section we provide an application of Theorem \ref{main}. We consider the circular weight functions of the form
$$\varphi(z)=-\alpha\log|z|+|z|^{\beta}$$ where $\beta>\alpha\geq 0.$ Then it follows from \cite[pp. 245]{SaffTotik}) that
$$U_{\varphi}(z)= \begin{cases}
\varphi(r_0) & |z| \leq r_0\\
\varphi(z) & r_0<|z|<R_0\\
\log|z|+\varphi(R_0)-\log R_0 & |z|\geq R_0
\end{cases}$$
 where $r_0=(\frac{\alpha}{\beta})^{\frac{1}{\beta}}$ and $R_0=(\frac{1+\alpha}{\beta})^{\frac{1}{\beta}}.$ In particular, the weighted equilibrium measure
$$d\mu_e=\frac{\mathbbm{1}_{D_{\varphi}}}{2\pi}\beta^2r^{\beta-1}drd\theta,\ \ z=re^{i\theta}$$ is supported on the annulus $D_{\varphi}=\{r_0\leq |z|\leq R_0\}$. In this setting a random polynomial is of the form
$$f_n(z)=\sum_{j=0}^na^n_jc^n_jz^j$$ where 
$$c^n_j=(\int_{\C}|z|^{2j+2n\alpha}e^{-2n|z|^{\beta}}dz)^{-\frac12}$$ and $a^n_j$ are independent copies of a random variable $\zeta$ with mean zero, variance one and satisfying $\E|\zeta|^{2+\delta}<\infty$ for some $\delta>0.$

 Finally, it follows from Theorem \ref{main} that asymptotic number of real roots is given by
$$\lim_{n\to \infty}\frac{1}{\sqrt{n}}\E N_n^{\zeta}=\frac{\sqrt{2\beta}}{\pi}(\sqrt{\alpha+1}-\sqrt{\alpha}).$$
 \bibliographystyle{alpha}
\bibliography{biblio}

\begin{thebibliography}{LPX16}

\bibitem[Bay]{B9}
T.~Bayraktar.
\newblock Asymptotic normality of linear statistics of zeros of random
  polynomials.
\newblock {\em Proc. of Amer. Math. Soc.}, 145 (7): 2917-2929, 2017.

\bibitem[BD04]{BDi}
P.~Bleher and X.~Di.
\newblock Correlations between zeros of non-{G}aussian random polynomials.
\newblock {\em Int. Math. Res. Not.}, (46):2443--2484, 2004.

\bibitem[Ber09]{Berman}
R.~J. Berman.
\newblock Bergman kernels for weighted polynomials and weighted equilibrium
  measures of {$\Bbb C^n$}.
\newblock {\em Indiana Univ. Math. J.}, 58(4):1921--1946, 2009.

\bibitem[BL15]{BloomL}
T.~Bloom and N.~Levenberg.
\newblock Random polynomials and pluripotential-theoretic extremal functions.
\newblock {\em Potential Anal.}, 42(2):311--334, 2015.

\bibitem[BP32]{BPo}
A.~Bloch and G.~P{\'o}lya.
\newblock On the {R}oots of {C}ertain {A}lgebraic {E}quations.
\newblock {\em Proc. London Math. Soc.}, S2-33(1):102, 1932.

\bibitem[Chr91]{Christ}
M.~Christ.
\newblock On the {$\overline\partial$} equation in weighted {$L^2$} norms in
  {${\bf C}^1$}.
\newblock {\em J. Geom. Anal.}, 1(3):193--230, 1991.

\bibitem[Das71]{Das71}
M.~Das.
\newblock Real zeros of a random sum of orthogonal polynomials.
\newblock {\em Proc. Amer. Math. Soc.}, 27:147--153, 1971.

\bibitem[Del98]{Delin}
H.~Delin.
\newblock Pointwise estimates for the weighted {B}ergman projection kernel in
  {$\bold C^n$}, using a weighted {$L^2$} estimate for the
  {$\overline\partial$} equation.
\newblock {\em Ann. Inst. Fourier (Grenoble)}, 48(4):967--997, 1998.

\bibitem[DS89]{DeSt}
J.-D. Deuschel and D.~W. Stroock.
\newblock {\em Large deviations}, volume 137 of {\em Pure and Applied
  Mathematics}.
\newblock Academic Press, Inc., Boston, MA, 1989.

\bibitem[EK95]{EdK}
A.~Edelman and E.~Kostlan.
\newblock How many zeros of a random polynomial are real?
\newblock {\em Bull. Amer. Math. Soc. (N.S.)}, 32(1):1--37, 1995.

\bibitem[EO56]{EO}
P.~Erd{\"o}s and A.~C. Offord.
\newblock On the number of real roots of a random algebraic equation.
\newblock {\em Proc. London Math. Soc. (3)}, 6:139--160, 1956.

\bibitem[Ess68]{Esseen}
C.~G. Esseen.
\newblock On the concentration function of a sum of independent random
  variables.
\newblock {\em Z. Wahrscheinlichkeitstheorie und Verw. Gebiete}, 9:290--308,
  1968.

\bibitem[ET50]{ET}
P.~Erd{\"o}s and P.~Tur{\'a}n.
\newblock On the distribution of roots of polynomials.
\newblock {\em Ann. of Math. (2)}, 51:105--119, 1950.

\bibitem[IM71a]{IM1}
I.~A. Ibragimov and N.~B. Maslova.
\newblock The mean number of real zeros of random polynomials. {I}.
  {C}oefficients with zero mean.
\newblock {\em Teor. Verojatnost. i Primenen.}, 16:229--248, 1971.

\bibitem[IM71b]{IM2}
I.~A. Ibragimov and N.~B. Maslova.
\newblock The mean number of real zeros of random polynomials. {II}.
  {C}oefficients with a nonzero mean.
\newblock {\em Teor. Verojatnost. i Primenen.}, 16:495--503, 1971.

\bibitem[Kac43]{Kac}
M.~Kac.
\newblock On the average number of real roots of a random algebraic equation.
\newblock {\em Bull. Amer. Math. Soc.}, 49:314--320, 1943.

\bibitem[KZ]{KZ2}
Z~Kabluchko and D~Zaporozhets.
\newblock Universality for zeros of random analytic functions.
\newblock {\em preprint arXiv 1205.5355}.

\bibitem[Lin01]{Lindholm}
N.~Lindholm.
\newblock Sampling in weighted {$L^p$} spaces of entire functions in {${\Bbb
  C}^n$} and estimates of the {B}ergman kernel.
\newblock {\em J. Funct. Anal.}, 182(2):390--426, 2001.

\bibitem[LO43]{LO}
J.~E. Littlewood and A.~C. Offord.
\newblock On the number of real roots of a random algebraic equation. {III}.
\newblock {\em Rec. Math. [Mat. Sbornik] N.S.}, 12(54):277--286, 1943.

\bibitem[LO45]{LO1}
J.~E. Littlewood and A.~C. Offord.
\newblock On the distribution of the zeros and {$a$}-values of a random
  integral function. {I}.
\newblock {\em J. London Math. Soc.}, 20:130--136, 1945.

\bibitem[LO48]{LO2}
J.~E. Littlewood and A.~C. Offord.
\newblock On the distribution of zeros and {$a$}-values of a random integral
  function. {II}.
\newblock {\em Ann. of Math. (2)}, 49:885--952; errata 50, 990--991 (1949),
  1948.

\bibitem[LPX15]{LPX}
D.S. Lubinsky, I.E. Pritsker, and X.~Xie.
\newblock Expected number of real zeros for random linear combinations of
  orthogonal polynomials.
\newblock {\em arXiv preprint arXiv:1503.06376}, 2015.

\bibitem[LPX16]{LPX1}
D.~S Lubinsky, I.~E. Pritsker, and X.~Xie.
\newblock Expected number of real zeros for random orthogonal polynomials.
\newblock {\em Math. Proc. Camb. Phil. Soc}, 2016.

\bibitem[PX15]{PX}
I.~E. Pritsker and X.~Xie.
\newblock Expected number of real zeros for random {F}reud orthogonal
  polynomials.
\newblock {\em J. Math. Anal. Appl.}, 429(2):1258--1270, 2015.

\bibitem[ST97]{SaffTotik}
E.~B. Saff and V.~Totik.
\newblock {\em Logarithmic potentials with external fields}, volume 316.
\newblock Springer-Verlag, Berlin, 1997.
\newblock Appendix B by Thomas Bloom.

\bibitem[ST05]{SoT2}
M.~Sodin and B.~Tsirelson.
\newblock Random complex zeroes. {III}. {D}ecay of the hole probability.
\newblock {\em Israel J. Math.}, 147:371--379, 2005.

\bibitem[SZZ08]{SZZ}
B.~Shiffman, S.~Zelditch, and S.~Zrebiec.
\newblock Overcrowding and hole probabilities for random zeros on complex
  manifolds.
\newblock {\em Indiana Univ. Math. J.}, 57(5):1977--1997, 2008.

\bibitem[TV15]{TaoVu2}
T.~Tao and V.~Vu.
\newblock Local universality of zeroes of random polynomials.
\newblock {\em Int. Math. Res. Not. IMRN}, (13):5053--5139, 2015.

\bibitem[Van15]{vanderbei}
R.~J Vanderbei.
\newblock The complex roots of random sums.
\newblock {\em ArXiv:1508.05162}, 2015.

\end{thebibliography}
\end{document}